\documentclass[11pt]{amsart}
\usepackage{mabliautoref}
\usepackage{amssymb,amsthm,amsmath}
\RequirePackage[dvipsnames,usenames]{xcolor}
\usepackage{hyperref}
\usepackage[all]{xy}
\usepackage{tikz}

\hypersetup{
bookmarks,
bookmarksdepth=3,
bookmarksopen,
bookmarksnumbered,
pdfstartview=FitH,
colorlinks,backref,hyperindex,
linkcolor=Sepia,
anchorcolor=BurntOrange,
citecolor=MidnightBlue,
citecolor=OliveGreen,
filecolor=BlueViolet,
menucolor=Yellow,
urlcolor=OliveGreen
}

\theoremstyle{remark}

\DeclareMathAlphabet{\mathchanc}{OT1}{pzc}%
                                 {m}{it}

\newcommand{\mcH}{\mathchanc{H}}

\newcommand{\mcm}{\mathchanc{m}}

\newcommand{\mco}{\mathchanc{o}}

\newcommand{\bC}{\mathbb{C}}

\newcommand{\bZ}{\mathbb{Z}}

\newcommand{\scr}{\mathcal}

\newcommand{\sC}{\scr{C}}

\newcommand{\sE}{\scr{E}}
\newcommand{\sF}{\scr{F}}

\newcommand{\sL}{\scr{L}}

\newcommand{\sO}{\scr{O}}

\newcommand{\op}{\overline{p}}

\newcommand{\sHom}[0]{{\mcH\mco\mcm}}

\DeclareMathOperator{\red}{red}

\newcommand{\factor}[2]{\left. \raise 2pt\hbox{\ensuremath{#1}} \right/
        \hskip -2pt\raise -2pt\hbox{\ensuremath{#2}}}

\makeatletter
\renewcommand\subsection{
  \renewcommand{\sfdefault}{pag}
  \@startsection{subsection}%
  {2}{0pt}{.8\baselineskip}{.4\baselineskip}{\raggedright
    \sffamily\itshape\small\bfseries
  }}
\renewcommand\section{
  \renewcommand{\sfdefault}{phv}
  \@startsection{section} %
  {1}{0pt}{\baselineskip}{.8\baselineskip}{\centering
    \sffamily
    \scshape
    \bfseries
}}
\makeatother

\renewcommand{\theenumi}{\arabic{enumi}}

\usepackage[left=1.02in,top=0.9in,right=1.02in,bottom=0.9in]{geometry}

\title{Semi-negativity of Hodge bundles associated to Du Bois families}
\author{Zsolt Patakfalvi}

\begin{document}

\maketitle

\begin{abstract}
In this note we show that the sheaf $R^1 f_* \sO_X$ is an anti-nef vector bundle (i.e., its dual is nef), where $f : X \to Y$ is a family of Du Bois schemes of pure dimension.
\end{abstract}

\section{Introduction}

In this note we show that the sheaf $R^1 f_* \sO_X$ is an anti-nef vector bundle (i.e., its dual is nef), where $f : X \to Y$ is a family of Du Bois schemes of pure dimension $n$.  Note that by \cite{Kollar_Kovacs_Log_canonical_singularities_are_Du_Bois}, $R^1 f_* \sO_X$ is  known to be a vector bundle, so our contribution is proving anti-nefness. This statement is the generalization of the classical result stating that if $f$ is smooth, then the Hodge metric on $R^1 f_* \sO_X$ has semi-negative curvature \cite{Griffiths_Periods_of_integrals}. For the definition and properties of Du Bois singularities we refer to \cite{Kovacs_Schwede_Hodge_theory_meets_the_minimal_model_program_a_survey_of_log_canonical_and_Du_Bois_singularities}, and here we only note that they can be viewed as the largest class of singularities where vanishing theorems hold 
\cite[9.12, 12.7]{Kollar_Shafarevich_maps_and_automorphic_forms}.

We should also point out that  the semi-negativity of $R^1 f_* \sO_X$ in this case is related but not equivalent to the widely investigated semi-positivity of $R^{n-1} f_* \omega_{X/Y}$ (e.g., \cite[Theorem 1.4]{Fujino_Fujisawa_Variations_of_mixed_Hodge_structures_and_semi_positivity_theorems}). In fact, the semi-negativity of $R^1 f_* \sO_X$ is equivalent in this case to the semi-positivity of $R^{-1} f_* \omega_{X/Y}^{\bullet}$, which sheaf is not equal to $R^{n-1} f_* \omega_{X/Y}$ simply because $\omega_{X/Y}^{\bullet}$ is not isomorphic to $\omega_{X/Y}[n]$. Indeed, $\omega_{X/Y}[n]$ is the $-n$-th cohomology sheaf of $\omega_{X/Y}^{\bullet}$, but in general $\omega_{X/Y}^{\bullet}$ has many other higher cohomology sheaves. (Recall that $\omega_{X/Y}$ is the $-n$-th cohomology sheaf of $\omega_{X/Y}^{\bullet}$, as defined in \autoref{subsec:notation})

\begin{theorem}
\label{thm:semi_negative}
If $f: X \to Y$ is a flat, projective family of connected, Du Bois schemes of pure dimension $n$ over $\bC$,  then   $R^1 f_* \sO_X$ is an anti-nef or equivalently $R^1 f_* \omega_{X/Y}^{\bullet}$ is a nef vector bundle. 
\end{theorem}

\begin{remark}
\label{rem:why_not_Fujino}
One would be tempted to use directly the available semipositivity results for reducible fiber spaces \cite{Fujino_Fujisawa_Variations_of_mixed_Hodge_structures_and_semi_positivity_theorems}, \cite{Kawamata_Semipositivity_theorem_for_reducible_algebraic_fiber_spaces} to prove \autoref{thm:semi_negative}. However, the author does not see a way of doing it, due to certain assumptions on the strata and monodromies in \cite{Fujino_Fujisawa_Variations_of_mixed_Hodge_structures_and_semi_positivity_theorems} and \cite{Kawamata_Semipositivity_theorem_for_reducible_algebraic_fiber_spaces}. Instead, we use an injectivity theorem for Du Bois schemes.
\end{remark}

The  main ingredients in proving \autoref{thm:semi_negative} are \autoref{thm:injectivity_Du_Bois} and \autoref{cor:generic_global_generation}, which are shown in \autoref{subsec:injectivity_surjectivity_Du_Bois}. Note that \autoref{thm:injectivity_Du_Bois} was shown in \cite[Thm 9.12]{Kollar_Shafarevich_maps_and_automorphic_forms} for normal schemes. Though we believe the arguments of \cite[Thm 12.10]{Kollar_Shafarevich_maps_and_automorphic_forms} can be generalized to non-normal schemes, for the convenience of the reader we include a different proof here.

\begin{theorem}
\label{thm:injectivity_Du_Bois}
If $X$ is a projective, Du Bois scheme, $N>0$ an integer, $\sL$ a line bundle on $X$, such that $\sL^N$ is globally generated and $F$ a general effective divisor of $\sL^N$, then the natrual map
\begin{equation}
H^i(X,\omega_X^{\bullet} \otimes \sL ) \to H^i(X,\omega_X^{\bullet} \otimes \sL(F) ) 
\end{equation}
is injective.
\end{theorem}

For the next statement, recall that a sheaf $\sF$ on an integral scheme $Y$ is generically globally generated, if global section of $\sF$ generate the stalk of $\sF$ at the generic point of $Y$.

\begin{corollary}
\label{cor:generic_global_generation}
Let $f : X \to Y$ be a flat, projective Du Bois family over a smooth projective curve, $y_0 \in Y$ and $N>0$ such that $|NX_{y_0}|$ is base-point free. Then for any $i$, $R^i f_*(\omega_{X/Y}^{\bullet}) \otimes \omega_Y ((N+1)y_0 )$ is generically globally generated.
\end{corollary}

% Second, in \autoref{subsec:direct_decomposition_proof}, we show the following decomposition result, in the spirit of the celebrated article of Koll\'ar \cite{Kollar_Higher_direct_images_of_dualizing_sheaves_II}.
% 
% \begin{theorem}
% \label{thm:direct_sum}
% Let  $n \geq d \geq 2$ be arbitrary integers and $f : X \to Y$ a flat projective morphism with connected fibers, such that $X$ is a reduced scheme of pure dimension $n$  and $Y$ a smooth curve. Furthermore, assume that $X$ is $S_d$. Then  
% \begin{equation}
% \label{eq:induction_dir_sum:X2}
% R f_* \omega_X^{\bullet} \cong R^{\leq -d}f_* (\omega_X^{\bullet} ) \oplus \left( \bigoplus_{i>-d} R^{i} f_* \omega_X^{\bullet} [-i] \right) .
% \end{equation}
% \end{theorem}

\subsection{Notation}
\label{subsec:notation}

The base field is the field of complex numbers $\bC$. For a complex $\sC^{\bullet}$ of sheaves, $h^i(\sC^{\bullet})$ is the $i$-th cohomology sheaf of $\sC$. For a projective morphism $f : X \to Y$ such that $Y$ is of finite type over $\bC$, $\omega_{X/Y}^{\bullet}:= f^! \sO_Y$, where $f^!$ is the functor obtained in \cite[Corollary VII.3.4.a]{Hartshorne_Residues_and_duality}. If $f$ has equidimensional fibers of dimension $n$, then $\omega_{X/Y} := h^{-n} (\omega_{X/Y}^{\bullet})$. Every complex and morphism of complexes is considered in the derived category $D(qc/\_)$ of quasi-coherent sheaves up to the equivalences defined there.

\subsection{Acknowledgement}

The author is grateful to Bhargav Bhatt, J\'anos Koll\'ar and Karl Schwede for fruitful discussions about the article.

\section{The proof of semi-positivity}
\label{subsec:semi_positivity_proof}

Since nefness is checked on curves, proving \autoref{thm:semi_negative} for a  curve base turns out to be the main issue.
This is proved in \autoref{prop:nef}, assuming \autoref{cor:generic_global_generation}, which will be showed in \autoref{subsec:injectivity_surjectivity_Du_Bois}.  We conclude this section with the (short) proof of \autoref{thm:semi_negative} using \autoref{prop:nef}.

\begin{proposition}
\label{prop:nef}
If $f: X \to Y$ is a flat, projective family of connected, Du Bois  schemes of pure dimension $n$ over a smooth, projective curve,  then $R^{1} f_* \sO_X$ is an anti-nef or equivalently $R^{-1} f_* \omega_{X/Y}^{\bullet}$ is a nef vector bundle.
\end{proposition}

We will prove \autoref{prop:nef} at the end of this section, after listing a few lemmas.

\begin{lemma}
\label{lem:locally_free}
If $f: X \to Y$ is a flat, projective family with Du Bois fibers, then  $(R^i f_* \sO_X)^* \cong R^{-i} f_* \omega_{X/Y}^{\bullet}$.
\end{lemma}

\begin{proof}
By \cite[Theorem 7.8]{Kollar_Kovacs_Log_canonical_singularities_are_Du_Bois}, $R^i f_* \sO_X$ is locally free. Hence the following computation concludes our proof.
\begin{equation*}
R^{-i} f_* \omega_{X/Y}^{\bullet} \cong R^{-i} f_* R \sHom_X(\sO_X, \omega_{X/Y}^{\bullet}) 
\cong  \underbrace{R^{-i} \sHom_Y(R f_* \sO_X, \sO_Y)}_{\textrm{Grothendieck duality}} \cong 
\underbrace{(R^i f_* \sO_X)^*}_{\parbox{125pt}{ \tiny $R^i f_* \sO_X$ is locally free, hence the adequate spectral squence degenerates}}
\end{equation*}
 
\end{proof}

Since $\omega_{X/Y}^{\bullet}$ is the main object of \autoref{prop:nef} for fibrations $X \to Y$ that are not necessarily Cohen-Macaulay, we need the following well-known technical lemma. For a proof we refer to either \cite[Theorem 5.4]{Neeman_The_Grothendieck_duality_theorem} or \cite[III, Prop 8.8]{Hartshorne_Residues_and_duality}. The most important consequence is stated in \autoref{lem:dualizing_complex}, a formula relating the relative and absolute dualizing complexes. It turns out that, at least over Gorenstein bases, nothing surprising happens. 

\begin{lemma}
\label{lem:f_upper_shriek}
If $f : X \to Y$ is a flat, projective morphism between projective schemes, then for every $\sC^{\bullet} \in D(X)$,
\begin{equation*}
f^!(\sC^{\bullet}) \cong Lf^* (\sC^{\bullet}) \displaystyle\otimes_L f^{!} \sO_Y. 
\end{equation*}
\end{lemma}

% \begin{proof}
% For a projective morphism $f$,  Neeman's \cite{Neeman_The_Grothendieck_duality_theorem} and Hartshorne's definition \cite{Hartshorne_Residues_and_duality} of $f^!$   agree, since both are right adjoint functors of $Rf_*$. Hence we may use the results of \cite{Neeman_The_Grothendieck_duality_theorem} to prove the lemma. By \cite[Theorem 5.4]{Neeman_The_Grothendieck_duality_theorem}, it is enough to show that $f^!$ commutes with coproducts. Fix an ample line bundle $\sL$ on $X$. By the discussion of \cite[Example 1.10]{Neeman_The_Grothendieck_duality_theorem} for every $M \in \bZ$, $\{ \sL^m[n] | m, n \in \bZ, m>M\}$ is a compact generating set for $D(qc/X)$. Fix $M$ such that $H^i(X_y, \sL^m)=0$ for all $m>M$ and all $y \in Y$. Then for every $m >M$, $R f_* (\sL^m[n] )$ is supported only in cohomological degree $-n$ and furthermore with locally free cohomology sheaf according to \cite[Theorem 12.11]{Hartshorne_Algebraic_geometry}. In particular, it is a compact object of $D(qc/Y)$ 
% \cite[Example 1.10]{Neeman_The_Grothendieck_duality_theorem} (to be precise in \cite[Example 1.10]{Neeman_The_Grothendieck_duality_theorem}, it is only stated that  line bundles are compact, but verbatim the same proof works for a locally free sheaves, by replacing inverse with 
% dual). Hence $R f_* (\sF)$ is compact for every element $\sF$ of the generating set $\{ \sL^m[n] | m, n \in \bZ, m>M\}$ of $D(qc/X)$. Therefore, by \cite[Theorem 5.1]{Neeman_The_Grothendieck_duality_theorem} $f^!$ commutes with coproduct, which finishes our proof.
% \end{proof}

\begin{lemma}
\label{lem:dualizing_complex}
If $f : X \to Y$ is a flat projective morphism between projective schemes with Gorenstein base of pure dimension $d$, then
\begin{equation*}
\omega_{X/Y}^{\bullet} \otimes f^* \omega_Y[d] \cong \omega_X^{\bullet}
\end{equation*}
\end{lemma}

% \begin{proof}
% \begin{equation*}
% \omega_X^{\bullet} \cong f^! \omega_Y^{\bullet} \cong f^! \omega_Y[d] \cong 
% \underbrace{f^! \sO_Y  \otimes f^* \omega_Y[d]}_{\textrm{\autoref{lem:f_upper_shriek} and flatness of $f$ and $\omega_Y$}}
% \cong \omega_{X/Y}^{\bullet} \otimes f^* \omega_Y[d]
% \end{equation*}
% \end{proof}

We need a third lemma as well about the behavior of relative dualizing complexes, for which we introduce first some notation.

\begin{notation}
\label{notation:product}
For a morphism $f : X \to Y$ of schemes, define
\begin{equation*}
X^m_Y :=\underbrace{X \times_Y X \times_Y \dots \times_Y X}_{\textrm{$m$
times}} .
\end{equation*}
and $f^m_Y : X^m_Y \to Y$ the base morphism.  
In most cases, when $Y$ is obvious from the context, we omit $Y$ from our notation. We denote then the $i$-th projection morphisms $X^m \to X$ by $p_i $.
\end{notation}

\begin{lemma}
\label{lem:product_dualizing_complex}
Using \autoref{notation:product}, if $f : X \to Y$ is a flat projective morphism of projective schemes, then 
\begin{equation*}
\omega_{X^m/Y}^{\bullet} \cong \bigotimes_L {}_{i=1}^m Lp_i^* \omega_{X/Y}^{\bullet}. 
\end{equation*}
\end{lemma}

\begin{proof}
The statement is vacuous for $m=1$. For $m>1$ we prove by induction. By the induction hypothesis
\begin{equation}
\label{eq:product_dualizing_complex:hypothesis}
\omega_{X^{m-1}/Y}^{\bullet} \cong \bigotimes_L {}_{i=1}^{m-1} L \op_i^* \omega_{X/Y}^{\bullet},
\end{equation}
where $\op_i$ is the $i$-th projection $X^{m-1} \to X$. Let $q: X^m \to X^{m-1}$  be the projection on the first $m-1$ factors. Then  the following computation concludes our proof.
\begin{multline*}
%\label{eq:product_dualizing_complex:isomorphism}
\omega_{X^m/Y}^{\bullet} \cong 
\underbrace{Lq^* \omega_{X^{m-1}/Y}^{\bullet} \otimes_L \omega_{X^m/X^{m-1}}^{\bullet}}_{\textrm{\autoref{lem:f_upper_shriek}}} 
\cong  
\underbrace{Lq^* \left( \bigotimes_L {}_{i=1}^{m-1} L \op_i^* \omega_{X/Y}^{\bullet} \right)   \otimes_L \omega_{X^m/X^{m-1}}^{\bullet}}_{\textrm{\autoref{eq:product_dualizing_complex:hypothesis}}}
\\ \cong  
\underbrace{\left(\bigotimes_L {}_{i=1}^{m-1} Lp_i^* \omega_{X/Y}^{\bullet} \right)    \otimes_L Lp_n^* \omega_{X/Y}^{\bullet}}_{\textrm{$L q^* L\op_i^* \cong L p_i^*$ and flat base change \cite[Theorem 8.7.5]{Hartshorne_Residues_and_duality}}}
 \cong  \bigotimes_L {}_{i=1}^m Lp_i^* \omega_{X/Y}^{\bullet}.
\end{multline*}
\end{proof}

Having finished the lemmas about the relative dualizing complex, we need two more auxiliary statements used in the proof of \autoref{prop:nef}.

\begin{lemma}
\label{lem:generic_global_generation_nef}
If $\sF$ is a vector bundle on a smooth projective curve $Y$ and $\sL$ is a line bundle such that for every $m>0$, $S^m(\sF) \otimes \sL$ is generically globally generated, then $\sF$ is nef. 
\end{lemma}

\begin{proof}
Take a finite cover $\tau : Z \to Y$ by a smooth curve and a quotient line bundle $\sE$ of $\tau^* \sF$. Since $S^m(\sF) \otimes \sL$ is generically globally generated, so is $S^m(\tau^* \sF) \otimes \tau^* \sL$ and hence  $\sE^m \otimes \tau^* \sL$ as well. Therefore $m \deg (\sE) + \deg (\tau^* \sL) \geq 0$ for all $m>0$. In particular then $\deg (\sE) \geq 0$. Since this is true for arbitrary $\tau$ and $\sE$, $\sF$ is nef indeed.
\end{proof}

\begin{proposition}
\label{prop:direct_sum_initial}
If $f : X \to Y$ a flat projective morphism with connected fibers, such that $Y$ has rational singularities, then $Rf_* \omega_{X/Y}^{\bullet} \cong R^{\leq -1} f_* \omega_{X/Y}^{\bullet} \oplus R^{0} f_* \omega_{X/Y}^{\bullet}$.
\end{proposition}

\begin{proof}
According to  \cite[Theorem 4.1.3]{Bhatt_Thesis}, the natural inclusion $\sO_Y \to Rf_* \sO_X$ splits. Since $f_* \sO_X \cong \sO_Y$ by the connectedness and flatness assumptions, this means that  $Rf_* \sO_X \cong \sO_Y \oplus R^{\geq 1} f_* \sO_X$. 
\begin{multline*}
R f_* \omega_{X/Y}^{\bullet} \cong R f_* R \sHom_X( \sO_X, \omega_{X/Y}^{\bullet}) \cong 
\underbrace{R \sHom_Y(R f_* \sO_X, \sO_Y) }_{\textrm{Grothendieck duality}}
\cong 
R \sHom_Y(\sO_Y \oplus R^{\geq 1} f_* \sO_X, \sO_Y) 
\\ 
%\cong 
%
%\underbrace{R \sHom_Y(\sO_Y \oplus R^{\geq 1} f_* \sO_X, \omega_Y[d]) }_{\textrm{$Y$ is Cohen-Macaulay of dimension $d$}}
%
\cong  
R \sHom_Y(R^{\geq 1} f_* \sO_X, \sO_Y) \oplus \sO_Y
\end{multline*}
Our proof is concluded by noting that $R^{\geq 1} f_* \sO_X$ is supported in cohomological degrees greater than $0$, and therefore $R \sHom_Y(R^{\geq 1} f_* \sO_X, \sO_Y)$ is supported in cohomological degrees less than $0$. 
\end{proof}

\begin{proof}[Proof of \autoref{prop:nef}]
According to \autoref{lem:locally_free}, we only have to prove  that $R^{-1} f_* (\omega_{X/Y}^{\bullet})$ is nef.
%Since the fibers of $f$ are reduced, so is $X$. By flatness and \cite[Corollary III.9.6]{Hartshorne_Algebraic_geometry}, $X$ is also of pure dimension $n+1$. 
%Furthermore,  by \cite[Lemma 4.2]{Patakfalvi_Schwede_Depth_of_F_singularities}, $X$ is $S_d$ and if $d=n$, then $X$ is $S_{n+1}$ (or equivalently Cohen-Macaulay).  Therefore, \autoref{thm:direct_sum} applies. 
%Also, by \autoref{lem:dualizing_complex} we may replace $\omega_X^{\bullet}$ in the statement of \autoref{prop:direct_sum_initial} by $\omega_{X/Y}^{\bullet}$ if we also shift the indices by $1$. 
% \begin{equation}
% That is, for all $i < d$ (or $i \leq n$ in the $d=n$ case) the following holds.
% \label{eq:nef:direct_sum}
% R f_* \omega_{X/Y}^{\bullet} \cong R^{\leq -i} f_* (\omega_{X/Y}^{\bullet} ) \oplus \left( \bigoplus_{l>-i} R^{l} f_* \left(\omega_{X/Y}^{\bullet} [-l] \right) \right) 
% \end{equation}
Fix an integer $m>0$  and consider the following stream of isomorphisms and surjections, using \autoref{notation:product}.
\begin{equation}
\label{eq:stable_family_global_generation:long}
\begin{split}
R^{-m} (f^m)_* (\omega_{X^m/Y}^{\bullet})  & 
\cong 
\underbrace{R^{-m} (f^m)_* \left( \bigotimes_L {}_{i=1}^m L p_i^* (\omega_{X/Y}^{\bullet} ) \right)}_{\textrm{\autoref{lem:product_dualizing_complex}}} \\
& \cong 
\underbrace{h^{-m} \left( \bigotimes_L {}_{j=1}^m R f_* (\omega_{X/Y}^{\bullet} )
\right)}_{\textrm{K\"unneth formula}} \\
& \cong 
\underbrace{h^{-m} \left( \bigotimes_L {}_{j=1}^m \left( R^{\leq -1} f_* (\omega_{X/Y}^{\bullet} ) \oplus  R^{0} f_* (\omega_{X/Y}^{\bullet} ) \right) 
\right)}_{\textrm{\autoref{prop:direct_sum_initial}}} \\
& \twoheadrightarrow h^{-m} \left( \bigotimes_L {}_{j=1}^m R^{ \leq -1} f_* (\omega_{X/Y}^{\bullet} ) 
\right)
 \\
& \cong \underbrace{\bigotimes_{i=1}^m R^{-1} f_* (\omega_{X/Y}^{\bullet} )}_{ \parbox{130pt}{\tiny  $\bigotimes_L$ is left derived, and $R^{-1} f_* (\omega_{X/Y}^{\bullet})$ is the highest non-zero cohomology sheaf of $R^{\leq -1} f_* (\omega_{X/Y}^{\bullet})$}}
%
%& \cong  \underbrace{\bigotimes_{i=1}^m R^{i} f_* (\omega_{X/Y} \otimes \sL)
%}_{\textrm{Lemma \ref{lem:splitting_Du_Bois}}} 
%
\\ &  \twoheadrightarrow S^m( R^{-1} f_* (\omega_{X/Y}^{\bullet} ))
 \\
\end{split}
\end{equation}
Fix any $y_0 \in Y$ and $N \in \bZ$, such that $|N y_0|$ is base-point free. By \autoref{cor:generic_global_generation}, $R^{-m} (f^m)_* (\omega_{X^m/Y}^{\bullet}) \otimes \omega_Y((N+1)y_0)$ is generically globally generated. Hence by \autoref{eq:stable_family_global_generation:long}, So is $S^m( R^{-1} f_* (\omega_{X/Y}^{\bullet} )) \otimes \omega_Y((N+1)y_0)$. Therefore, by \autoref{lem:generic_global_generation_nef}, $R^{-1} f_* (\omega_{X/Y}^{\bullet} )$ is nef, which concludes our proof.
\end{proof}

\begin{proof}[Proof of \autoref{thm:semi_negative}]
By \autoref{lem:locally_free}, the  statement on $R^1 f_* \sO_X$ and $R^{-1} \omega_{X/Y}^{\bullet}$ is  equivalent indeed. By \cite[Theorem 7.8]{Kollar_Kovacs_Log_canonical_singularities_are_Du_Bois}, $R^1 f_* \sO_X$ is combatible with arbitrary base-change. Furthermore, since nefness is decided on curves, we may assume that $Y$ is a smooth curve. However, then using \autoref{lem:locally_free} again, \autoref{prop:nef} concludes our proof.
\end{proof}

\section{Injectivity and surjectivity for Du Bois schemes}
\label{subsec:injectivity_surjectivity_Du_Bois}

Here we prove \autoref{thm:injectivity_Du_Bois} and \autoref{cor:generic_global_generation}.  
%Theorem \ref{thm:injectivity_Du_Bois} is also well known for rational singularities. However, proving it for Du Bois singularities is by far a non-trivial extension. It is a product of the work of many, most notably Fujino \cite[Theorem 2.38]{Fujino_Introduction_to_the_log_minimal_model_program_for_log_canonical_pairs} and Schwede \cite[Theorem 4.6]{Schwede_A_simple_characterization_of_Du_Bois_singularities}.  The main trick of the proof was communicated to the author by Karl Schwede.
  
\begin{proof}[Proof of \autoref{thm:injectivity_Du_Bois}]
Consider a closed embedding of $X$ into a smooth scheme $Y$, and let $ \rho : Z  \to Y$ be an embedded log-resolution of $(Y,X)$, which is isomorphism on $Y \setminus X$. Set $E:=\rho^{-1}(X)_{\red}$ and $\pi:= \rho|_E$. By  \cite[Theorem 4.6]{Schwede_A_simple_characterization_of_Du_Bois_singularities}, the natural homomorphism $\sO_X \to R \pi_* \sO_E$ is quasi-isomorphism. This yields the following isomorphisms.
\begin{equation}
\label{eq:injectivity_Du_Bois:dualizing}
R \pi_* \omega_E^{\bullet} \cong R \pi_* R \sHom_E( \sO_E, \omega_E^{\bullet}) 
\cong \underbrace{R \sHom_X( R \pi_* \sO_E , \omega_X^{\bullet}  ) }_{\textrm{Grothendieck-duality}}
\cong 
\underbrace{\omega_X^{\bullet}}_{\textrm{\cite[Theorem 4.6]{Schwede_A_simple_characterization_of_Du_Bois_singularities}}}
\end{equation}
\begin{multline}
\label{eq:injectivity_Du_Bois:hypercohomology}
H^{i + \dim E} (E, \omega_{E} \otimes \pi^* \sL ) \cong 
\underbrace{H^i (E, \omega_{E}^{\bullet} \otimes \pi^* \sL )}_{\textrm{$E$  Gorenstein, hence $\omega_E^{\bullet} \cong \omega_E [ \dim E]$}} 
\\ \cong
\underbrace{H^i(Y, R \pi_* (\omega_{E}^{\bullet} \otimes \pi^* \sL )) }_{\textrm{Grothendieck spectral sequence}} 
 \cong 
\underbrace{H^i(Y, R \pi_* (\omega_{E}^{\bullet}) \otimes \sL ) }_{\textrm{projection formula}} 
\cong 
\underbrace{H^i( \omega_{X}^{\bullet} \otimes \sL ) }_{\textrm{\autoref{eq:injectivity_Du_Bois:dualizing}}} 
\end{multline}
Furthermore, by replacing $\sL$ in \autoref{eq:injectivity_Du_Bois:hypercohomology} with $\sL(F)$, one obtains that
\begin{equation}
\label{eq:injectivity_Du_Bois:hypercohomology2}
H^{i+ \dim E} (E, \omega_{E} \otimes \pi^* \sL(F) ) \cong H^i( \omega_{X}^{\bullet} \otimes \sL(F) ) ,
\end{equation}
and \autoref{eq:injectivity_Du_Bois:hypercohomology} and \autoref{eq:injectivity_Du_Bois:hypercohomology2} are compatible with the natural maps induced by $\sL \to \sL(F)$. Hence, by setting $j=i + \dim E$, it is enough to prove that the natural homomorphisms
\begin{equation}
\label{eq:injectivity_Du_Bois:is_it_injective}
H^j(E,\omega_E \otimes \pi^* \sL ) \to H^j(E,\omega_E \otimes \pi^* \sL(  F) )
\end{equation}
are injective for every $j$. Note at this point that since $\pi^* F$ is a general member of a base-point free linear system, it does not contain any strata of $E$. In particular then  \cite[Theorem 2.38]{Fujino_Introduction_to_the_log_minimal_model_program_for_log_canonical_pairs} (setting $X:= E$, $D':=0$, $D:=\pi^* F$, $H$ be any divisor such that $\sO_E(H) \cong \pi^* \sL$, $t:=N$, $B:=0$, $S:=0$)  implies the injectivity of \autoref{eq:injectivity_Du_Bois:is_it_injective}.
\end{proof}

\begin{remark}
\autoref{thm:injectivity_Du_Bois} also follows from the arguments of \cite[Theorem 9.12]{Kollar_Shafarevich_maps_and_automorphic_forms} using \cite[Corollary 7.7]{Kollar_Kovacs_Log_canonical_singularities_are_Du_Bois}. Unfortunately, \cite[Theorem 9.12]{Kollar_Shafarevich_maps_and_automorphic_forms} is stated for irreducible $X$, hence we included a full proof of \autoref{thm:injectivity_Du_Bois}.
\end{remark}

To prove \autoref{cor:generic_global_generation}, we need two more lemmas. The proof of the first one is well-known exercise with exact triangles, hence we omit it. 

\begin{lemma}
\label{lem:adjunction}
If $X$ is a quasi-projective scheme and $H$ an effective Cartier divisor on it, then there is an adjunction exact triangle as follows.
\begin{equation*}
\xymatrix{
\omega_X^{\bullet} \ar[r] & \omega_X^{\bullet}(H) \ar[r] & \omega_H^{\bullet}[1] \ar[r]^{+1} & 
} 
\end{equation*}

\end{lemma}

% \begin{proof}
% If $\iota : H \to X$ is the embedding morphism, then
% \begin{equation*}
% \omega_H^{\bullet} \cong \iota^! \omega_X^{\bullet} =  R \sHom_H(\sO_H, \iota^! \omega_X^{\bullet} )  \cong 
% %
% \underbrace{R \sHom_X(\sO_H, \omega_X^{\bullet} )}_{\textrm{by Grothendieck duality}} 
% %
% .
% \end{equation*}
% Consider then the exact sequence
% \begin{equation*}
% %\label{eq:hypersurface}
% \xymatrix{
% 0 \ar[r] & \sO_X(-H)  \ar[r] & \sO_X \ar[r] & \sO_H \ar[r] & 0 ,
% }
% \end{equation*}
% and apply $R \sHom_X( \_ , \omega_X^{\bullet})$ to it:
% \begin{equation}
% \label{eq:ext_applied}
% \xymatrix{
%  \omega_H^{\bullet} \cong  R \sHom_X(\sO_H, \omega_X^{\bullet} ) \ar[r] & \omega_X^{\bullet} \ar[r] & \omega_X^{\bullet}(H) \ar[r]^-{+1} &
% }
% \end{equation}
% Rotating \autoref{eq:ext_applied} yields the statement of the lemma.
% \end{proof}

\begin{lemma}
\label{lem:surjective_DuBois}
Let $f : X \to Y$ be a flat, projective Du Bois family over a smooth projective curve, $y_0 \in Y$, $N>0$ such that $|NX_{y_0}|$ is base-point free and $A \in |NX_{y_0}|$ a generic element. Then for any $i$ and any $y \in Y$ such that $X_y \subseteq A$, the natural map $\alpha$ in the following diagram  is surjective.
\begin{equation}
\label{eq:surjective_DuBois:surjection}
\xymatrix{
H^i(X,\omega_{X/Y}^{\bullet} \otimes f^* \omega_Y ((N+1)X_{y_0} )) \cong H^i(X,\omega_X^{\bullet} (A+ X_{y_0}  )[-1]) \ar[r] \ar[dr]_{\alpha} &  H^i(A,\omega_{A}^{\bullet}(X_{y_0}) ) \cong H^i(A,\omega_{A}^{\bullet}) \ar@{->>}[d]  \\
&  H^i(X_y,\omega_{X_y}^{\bullet}) ,
}
\end{equation}
Here the horizontal homomorphism is induced by the adjunction map $\omega_X^{\bullet}(A)[-1] \to \omega_A^{\bullet}$ of \autoref{lem:adjunction}.
\end{lemma}

\begin{proof}
The vertical arrow of \autoref{eq:surjective_DuBois:surjection} is surjective because $X_y$ is a connected component of $A$. Therefore, it is enough to prove that the horizontal arrow of \autoref{eq:surjective_DuBois:surjection} is  surjective. However, then equivalently we may also show that
\begin{equation}
\label{eq:surjective_DuBois:injection}
H^i(X,\omega_X^{\bullet} (X_{y_0})[-1]) \to H^i(X,\omega_X^{\bullet} (X_{y_0}+A)[-1]) 
\end{equation}
is injective for all $i$. Note at this point that by \cite[Main Theorem]{Kovacs_Schwede_Du_Bois_singularities_deform}, $X$ itself is Du Bois. Hence, \autoref{eq:surjective_DuBois:injection} follows from \autoref{thm:injectivity_Du_Bois}.
\end{proof}

\begin{proof}[Proof of \autoref{cor:generic_global_generation}]
For any $y  \in Y$, 
\begin{multline}
\label{eq:surjective_DuBois:equation}
\dim_{k(y)}  \left(R^{i} f_* \omega_{X/Y}^{\bullet} \otimes k(y) \right) 
= 
\underbrace{\dim_{k(y)} \left(R^{-i} f_* \sO_X \otimes k(y) \right)}_{\textrm{\autoref{lem:locally_free}}} 
\\ =   
\underbrace{\dim_{k(y)} H^{-i}(X_y, \sO_{X_y})}_{\textrm{\cite[Theorem 7.8]{Kollar_Kovacs_Log_canonical_singularities_are_Du_Bois}}} 
=
\underbrace{\dim_{k(y)} H^{i} (X_y,\omega_{X_y}^{\bullet})}_{\textrm{Grothendieck duality}} . 
\end{multline}
Consider then the following diagram for a generic closed point $y \in Y$.
\begin{equation}
\label{eq:stable_family_global_generation:commutative}
\xymatrix{
 H^0(Y,R^i f_* (\omega_{X/Y}^{\bullet}) \otimes \omega_Y ((N+1)y_0 ) ) \ar[r]^{\beta} &
R^i f_* (\omega_{X/Y}^{\bullet}) \otimes \omega_Y ((N+1)y_0 )_y \ar[r]^-{\gamma} &  
H^{i}(X_y,\omega_{X_y}^{\bullet}  ) \\
H^{i}(X,\omega_{X/Y}^{\bullet} \otimes f^* \omega_Y ((N+1)y_0) ) \ar@{->>}[urr]_{\alpha} \ar[u]
}
\end{equation}
The arrow $\alpha$ is surjective by \autoref{lem:surjective_DuBois}, and by \autoref{eq:surjective_DuBois:equation} the two ends of $\gamma$ have the same dimensions over $\bC$. Hence $\beta$ also has to be surjective. This finishes our proof.
\end{proof}

\bibliographystyle{skalpha}
\bibliography{includeNice}
 
\end{document}